\documentclass[11pt, a4paper, reqno]{amsart}

\usepackage[T1]{fontenc}
\usepackage{palatino}
\usepackage{eucal}
\usepackage{esint}

\usepackage{amssymb}
\usepackage{amsmath}
\usepackage{amsthm}
\usepackage{amsfonts}
\usepackage{mathtools}
\usepackage{constants}
\usepackage{enumerate}
\usepackage[hidelinks]{hyperref}
\usepackage[nameinlink]{cleveref}
\usepackage[abbrev,backrefs]{amsrefs}
\usepackage[utf8]{inputenc}

\usepackage{tikz}

\numberwithin{equation}{section}

\theoremstyle{plain}
\newtheorem{theorem}{Theorem}
\newtheorem{proposition}{Proposition}[section]
\newtheorem{lemma}[proposition]{Lemma}
\newtheorem{corollary}[proposition]{Corollary}
\newtheorem*{claim}{Claim}

\theoremstyle{definition}
\newtheorem{remark}[proposition]{Remark}

\usepackage{xpatch}
\xpatchcmd{\proof}{\itshape}{\bfseries}{}{}

\newcommand{\R}{\mathbb{R}}
\newcommand{\N}{\mathbb{N}}

\newcommand{\loc}{_\mathrm{loc}}

\newcommand{\cd}{(D^2, L^1)}
\newcommand{\cl}[1]{(\Delta, L^{#1})}
\newcommand{\ch}{(\Delta, H^1)}

\newcommand{\Exp}{\mathrm{e}}

\newcommand{\cF}{\mathcal{F}}
\newcommand{\cH}{\mathcal{H}}
\newcommand{\cM}{\mathcal{M}}
\newcommand{\cN}{\mathcal{N}}
\newcommand{\cR}{\mathcal{R}}

\newcommand{\Nu}{\mathrm{N}}

\newcommand{\BMO}{\mathrm{BMO}}

\DeclarePairedDelimiter{\abs}{\lvert}{\rvert}
\DeclarePairedDelimiter{\bracks}{\lbrack}{\rbrack}
\DeclarePairedDelimiter{\norm}{\lVert}{\rVert}
\DeclarePairedDelimiter{\paren}{\lparen}{\rparen}
\DeclarePairedDelimiter{\set}{\{}{\}}

\DeclareMathOperator{\capt}{cap}
\DeclareMathOperator{\haus}{\cH}
\DeclareMathOperator{\diam}{diam}
\DeclareMathOperator{\dist}{dist}
\DeclareMathOperator{\Div}{div}
\DeclareMathOperator{\supp}{supp}
\DeclareMathOperator{\sgn}{sgn}

\let\d\relax
\newcommand{\d}{\mathop{}\!\mathrm{d}}

\newcommand{\NP}{\cN}
\newcommand{\RT}{\cR}

\newcommand{\Cs}{\ensuremath{C_{\arabic{cst@counter@normal}}}}

\title[Schr\"odinger operators involving singular potentials]{Schr\"odinger operators involving singular potentials and measure data}

\author{Augusto C. Ponce}
\address{
Augusto C. Ponce\hfill\break\indent
Universit{\'e} catholique de Louvain\hfill\break\indent
Institut de Recherche en Math{\'e}matique et Physique\hfill\break\indent
Chemin du Cyclotron 2, bte L7.01.02\hfill\break\indent
1348 Louvain-la-Neuve\hfill\break\indent
Belgium}
\email{Augusto.Ponce@uclouvain.be}

\author{Nicolas Wilmet}
\address{
Nicolas Wilmet\hfill\break\indent
Universit{\'e} catholique de Louvain\hfill\break\indent
Institut de Recherche en Math{\'e}matique et Physique\hfill\break\indent
Chemin du Cyclotron 2, bte L7.01.02\hfill\break\indent
1348 Louvain-la-Neuve\hfill\break\indent
Belgium}
\email{Nicolas.Wilmet@uclouvain.be}



\begin{document}

\begin{abstract}
We study the existence of solutions of the Dirichlet problem for the Schr\"odinger operator with measure data
\[
\left\{
\begin{alignedat}{2}
-\Delta u + Vu &= \mu && \quad \text{in \(\Omega\)}, \\
u &= 0 && \quad \text{on \(\partial\Omega\)}.
\end{alignedat}
\right.
\]
We characterize the finite measures \(\mu\) for which this problem has a solution for every nonnegative potential \(V\) in the Lebesgue space \(L^p(\Omega)\) with \(1 \le p \le \frac{N}{2}\). The full answer can be expressed in terms of the \(W^{2,p}\) capacity for \(p > 1\), and the \(W^{1,2}\) (or Newtonian) capacity for \(p = 1\). We then prove the existence of a solution of the problem above when \(V\) belongs to the real Hardy space \(H^1(\Omega)\) and \(\mu\) is diffuse with respect to the \(W^{2,1}\) capacity.
\end{abstract}

\maketitle

\section{Introduction and main results}

Let \(\Omega \subset \R^N\) be an open, bounded and smooth set in dimension \(N \ge 2\), and let \(V \in L^1(\Omega)\) be a nonnegative function. We address in this paper the question of existence of solutions of the linear Dirichlet problem with measure data
\begin{equation}
\label{eq:dirichletProblem}
\left\{
\begin{alignedat}{2}
-\Delta u + Vu &= \mu && \quad \text{in \(\Omega\)}, \\
u &= 0 && \quad \text{on \(\partial\Omega\)}.
\end{alignedat}
\right.
\end{equation}
The variational solution of this problem when \(\mu \in L^2(\Omega)\) can be obtained by a straightforward minimization of the associated energy functional
\[
E(v) = \frac{1}{2} \int_{\Omega} \paren*{\abs{\nabla v}^2 + V v^2} - \int_{\Omega} \mu v,
\]
which is bounded from below in \(W_{0}^{1, 2}(\Omega)\) since \(V\) is nonnegative. Using as a test function in \eqref{eq:dirichletProblem} a suitable approximation of \(\sgn u\), one deduces the absorption estimate
\begin{equation}
\label{eq:absorptionEstimate}
\norm{Vu}_{L^1(\Omega)} \le \norm{\mu}_{L^1(\Omega)}.
\end{equation}

When \(\mu \in L^1(\Omega)\), the functional \(E\) need not be bounded from below, but one can use an approximation argument with \(L^2\) functions to find a solution of \eqref{eq:dirichletProblem}, based on the linearity of the equation and the absorption estimate above; see \citelist{\cite{Brezis_Strauss:1973} \cite{Gallouet_Morel:1984}}. In this case, the solution \(u \in L^1(\Omega)\) satisfies \(Vu \in L^1(\Omega)\) and the functional identity:
\[
\int_{\Omega} u \paren*{-\Delta \zeta + V \zeta} = \int_{\Omega} \zeta \mu,
\]
for every \(\zeta \in C^\infty(\overline{\Omega})\) with \(\zeta = 0\) on \(\partial\Omega\). In the sequel, we denote by \(C_0^{\infty}(\overline{\Omega})\) the space of such test functions \(\zeta\). The right-hand side of this identity is well-defined even if \(\mu\) is merely a finite Borel measure; in this case the integral is interpreted as integration of \(\zeta\) with respect to \(\mu\). This is the notion of weak solution of the Dirichlet problem~\eqref{eq:dirichletProblem} which has been introduced by Littman, Stampacchia and Weinberger~\cite{Littman_Stampacchia_Weinberger:1963}*{Definition~5.1}.

In contrast with the \(L^1\) case, the existence of solutions of the Dirichlet problem~\eqref{eq:dirichletProblem} with measure data is more subtle. For example, in dimension \(N \ge 3\), the equation
\[
-\Delta u + \frac{u}{\abs{x}^{\alpha}} = \delta_0 \quad \text{in \(B_1\)},
\]
where \(B_1\) is the unit ball in \(\R^N\) centered at \(0\), has no solution in the sense of distributions when \(\alpha \ge 2\). Heuristically, \(u(x)\) behaves like \(\frac{1}{\abs{x}^{N-2}}\), as the fundamental solution of the Laplacian in a neighborhood of \(0\), and this is incompatible with the requirement that \(\frac{u}{\abs{x}^{\alpha}} \in L^1(B_1)\); see \Cref{prop:nonexistenceExample} below in the spirit of \cite{Benilan_Brezis:2003}*{Remark~A.4}. On the contrary, for \(\alpha < 2\), a solution does exist, and more generally Stampacchia~\cite{Stampacchia:1965}*{Th\'eor\`eme~9.1} proved that, for every nonnegative function \(V \in L^p(\Omega)\) with \(p > \frac{N}{2}\), the Dirichlet problem~\eqref{eq:dirichletProblem} has a solution for any finite measure \(\mu\).

These facts have a striking analogy with the Dirichlet problem associated to the semilinear equation
\begin{equation}
\label{eq:semilinearEquation}
-\Delta u + \abs{u}^{q-1} u = \mu \quad \text{in \(\Omega\)},
\end{equation}
motivated by the Thomas--Fermi model; see \cite{Lieb:1981}. B\'enilan and Brezis~\cite{Benilan_Brezis:2003} established the existence of solutions of the Dirichlet problem associated to \eqref{eq:semilinearEquation} for every \(1 \le q < \frac{N}{N-2}\). It was an open problem to characterize the class of finite measures for which \eqref{eq:semilinearEquation} has a solution when \(N \ge 3\) and \(q \ge \frac{N}{N-2}\). The answer has been provided by Baras and Pierre~\cite{Baras_Pierre:1984} in terms of a \(W^{2,q'}\) capacity; see also \cite{Vasquez:1983} for a counterpart in dimension \(2\). More precisely, a solution of \eqref{eq:semilinearEquation} exists if and only if \(\mu(K) = 0\) for every compact set \(K \subset \Omega\) with \(W^{2,q'}\) capacity zero.

In the same spirit, given \(1 \le p \le \frac{N}{2}\), we identify the measures for which the Dirichlet problem~\eqref{eq:dirichletProblem} has a solution for every nonnegative function \(V \in L^p(\Omega)\) and more generally in \(L\loc^p(\Omega)\); see \Cref{rem:remark} below. We state such a characterization in terms of the following capacity related to the Laplacian, defined for every compact set \(K \subset \Omega\) by
\[
\capt_{\cl{p}}(K;\Omega) = \inf \set*{\norm{\Delta \zeta}_{L^p(\Omega)}^p : \text{\(\zeta \in C_0^\infty(\overline{\Omega})\) is nonnegative and \(\zeta > 1\) in \(K\)}}.
\]
This is the content of
\begin{theorem}
\label{thm:characterization}
Let \(1 \le p \le \frac{N}{2}\) and \(\mu \in \cM(\Omega)\). Then the Dirichlet problem~\eqref{eq:dirichletProblem} has a solution for every nonnegative function \(V \in L^p(\Omega)\) if and only if \(\mu\) is diffuse with respect to the \(\cl{p}\) capacity.
\end{theorem}

We denote by \(\cM(\Omega)\) the Banach space of finite Borel measures in \(\Omega\), equipped with the norm
\[
\norm{\mu}_{\cM(\Omega)} = \abs{\mu}(\Omega).
\]
By \emph{diffuse}, we mean that \(\mu(K) = 0\) for every compact set \(K \subset \Omega\) with \(\capt_{\cl{p}}(K;\Omega) = 0\). Such a notion is the analogue of \textit{absolute continuity} from measure theory \cite{Ponce:2016}*{Proposition~14.7}.

The \(\cl{p}\) capacity is related to the more familiar Sobolev (or Bessel) capacities. More precisely, for every compact set \(K \subset \Omega\), we have
\begin{enumerate}[(i)]
\item \(\capt_{\cl{1}}(K;\Omega) = 0\) if and only if \(\capt_{W^{1,2}}(K) = 0\);
\item for every \(p > 1\), \(\capt_{\cl{p}}(K;\Omega) = 0\) if and only if \(\capt_{W^{2,p}}(K) = 0\).
\end{enumerate}
The second assertion is a consequence of the Cald\'eron--Zygmund elliptic \(L^p\) estimates \cite{Gilbarg_Trudinger:1983}*{Corollary~9.10}, while the first one follows from \cite{Brezis_Marcus_Ponce:2007}*{Theorem~4.E.1}; see also \cite{Ponce:2016}*{Proposition~12.2}. 

The existence of a solution of the Dirichlet problem~\eqref{eq:dirichletProblem} for \(V \in L^p(\Omega)\) is proved in \cite{Orsina_Ponce:2008} when \(p = 1\) using the method of sub and supersolutions.
The case \(p > 1\) is sketched in \cite{Orsina_Ponce:2016} using the approximation of diffuse measures by measures in the dual space \((W^{2,p}(\Omega) \cap W_0^{1,p}(\Omega))'\).
We propose in this paper a unified argument which covers both cases simultaneously, based on a strong approximation property of diffuse measures using the Hahn--Banach theorem in the spirit of \citelist{\cite{DalMaso:1983} \cite{Feyel_delaPradelle:1977}}.
For the reverse implication, we have been inspired by the proof of~\cite{Brezis_Marcus_Ponce:2007}*{Theorem~4.5} that treats the semilinear counterpart.
The conclusion of \Cref{thm:characterization} is also true for \(p > \frac{N}{2}\) by Stampacchia's existence result; in this case, every non-empty set has positive capacity, hence every measure is diffuse.

As the parameter \(p\) tends to \(1\), \Cref{thm:characterization} and Assertions~(i) and~(ii) combined suggest that the \(W^{1,2}\) and \(W^{2,1}\) capacities are equivalent. It turns out that this conclusion is \textit{not} correct. Indeed, D.~Adams proved in \cite{Adams:1988} that the \(W^{2,1}\) capacity vanishes on the same sets as the Hausdorff measure \(\haus^{N-2}\); this is the second-order analogue of a celebrated result by Fleming~\cite{Fleming:1960} concerning the \(W^{1,1}\) capacity and \(\haus^{N-1}\); see also \cite{Meyers_Ziemer:1977}. On the other hand, the \(W^{1,2}\) capacity vanishes on the same sets as the Newtonian capacity, and it is classically known \citelist{\cite{Evans_Gariepy:2015} \cite{Carleson:67}} that the latter capacity vanishes on every set of finite \(\haus^{N-2}\) measure.

The \(W^{2,1}\) capacity is thus squeezed between the \(W^{1,2}\) and \(W^{2,p}\) capacities for \(p > 1\). The sets where \(\capt_{W^{2,1}}\) vanishes in \(\Omega\) can alternatively be described in the spirit of \cite{Adams:1988} by replacing the Lebesgue \(L^p\) norm with the Hardy space \(H^1\) norm:
\[
\norm{u}_{H^1(\R^N)} = \norm{u}_{L^1(\R^N)} + \norm{\RT u}_{L^1(\R^N)},
\]
where \(\RT u\) is the Riesz transform of \(u\); see \Cref{sec:geometricInterpretation} below. 
For every compact set \(K \subset \R^{N}\), we then define
\[
\capt_{\ch}(K) = \inf \set*{\norm{\Delta \varphi}_{H^1(\R^N)} : \text{\(\varphi \in C_c^{\infty}(\R^N)\) is nonnegative and \(\varphi > 1\) in \(K\)}}.
\]
For convenience, we do not compute the capacity relative to \(\Omega\). We explain in \Cref{sec:geometricInterpretation} that, for every compact set \(K \subset \R^{N}\),
\[
\capt_{\ch}(K) = 0 \quad \text{if and only if} \quad \capt_{W^{2,1}}(K) = 0.
\]
In particular, a measure which is diffuse with respect to one capacity is also diffuse with respect to the other one.

Defining
\[
H^1(\Omega) = \set[\big]{f\lvert_{\Omega}\ : f \in H^1(\R^N)},
\]
we prove

\begin{theorem}
\label{thm:existence}
Let \(V \in H^1(\Omega)\) be a nonnegative function. Then the Dirichlet problem~\eqref{eq:dirichletProblem} has a solution for every measure \(\mu \in \cM(\Omega)\) which is diffuse with respect to the \(\ch\) capacity.
\end{theorem}

In dimension \(N = 2\), every measure is diffuse with respect to the \(\ch\) capacity. Thus, the Dirichlet problem~\eqref{eq:dirichletProblem} always has a solution with a nonnegative potential \(V \in H^1(\Omega)\). One might expect that \Cref{thm:existence} has a counterpart in the spirit of \Cref{thm:characterization}, but the converse is false in dimension \(N \ge 3\):

\begin{theorem}
\label{thm:existenceNondiffuse}
Suppose that \(N \ge 3\). Then there exists a positive measure \(\mu \in \cM(\Omega)\) with \(\capt_{\ch}(\supp \mu) = 0\) such that the Dirichlet problem~\eqref{eq:dirichletProblem} has a solution for every nonnegative function \(V \in H^1(\Omega)\).
\end{theorem}

The construction of \(\mu\) relies on the property that nonnegative functions in the Hardy space \(H^1(\Omega)\) are locally \(L \log L\) integrable. In fact, the situation is even more dramatic in the sense that, given any Orlicz space \(L\loc^{\Phi}(\Omega)\), one can find such a measure \(\mu\) so that the Dirichlet problem~\eqref{eq:dirichletProblem} has a solution for every nonnegative function \(V \in L\loc^{\Phi}(\Omega)\).

An alternative to this obstruction would be to consider all potentials of the form \(V = f^+\) with signed Hardy functions \(f \in H^1(\Omega)\). But in this case the counterpart of \Cref{thm:existence} fails since for such potentials it is not possible to solve the Dirichlet problem~\eqref{eq:dirichletProblem} for every diffuse measure; see \Cref{prop:nonexistence} below. It thus seems plausible that the characterization of diffuse measures via the Dirichlet problem~\eqref{eq:dirichletProblem} requires the use of signed potentials \(V \in H^1(\Omega)\). However, the operator \(-\Delta + V\) need not have a trivial kernel, even for \(V \in L^{\infty}(\Omega)\).

The paper is organized as follows. 
In Section~2, we prove the strong approximation property of diffuse measures that is used to establish the reverse implication of \Cref{thm:characterization}. 
In Section~3, we construct a suitable minimizing sequence of the \(\cl{p}\) capacity that is used in the direct implication. \Cref{thm:characterization} is proved in Section~4. 
In Section~5, we explain the connection between the \(\ch\) capacity and the \(\haus_{\infty}^{N-2}\) Hausdorff content due to Adams~\cite{Adams:1988}. 
We prove \Cref{thm:existence} in Section~6. 
For the sake of application, we then explain in Section~7 how \Cref{thm:existence} can be used to obtain a strong maximum principle for the Schr\"odinger operator along the lines of \cite{Orsina_Ponce:2016}. 
In Section~8, we prove an extension of \Cref{thm:existenceNondiffuse} in the setting of Orlicz spaces. 
In Section~9, we provide counterexamples to the existence of solutions of the Dirichlet problem~\eqref{eq:dirichletProblem} with measure data.

\section{Strong approximation of diffuse measures}
\label{sec:strongApproximation}

In this section, we prove a strong approximation property of diffuse measures, based on the Hahn--Banach theorem, which will be used in the proof of the existence of a solution of the Dirichlet problem~\eqref{eq:dirichletProblem} in the \(L^p\) setting (\Cref{thm:characterization}).

\begin{proposition}
\label{prop:strongApproximation}
Let \(1 \le p < \infty\) and \(\mu \in \cM(\Omega)\) be a nonnegative measure. If \(\mu\) is diffuse with respect to the \(\cl{p}\) capacity, then there exists a nondecreasing sequence \((\mu_n)_{n \in \N}\) of nonnegative measures in \(\cM(\Omega)\) with compact support in \(\Omega\) which satisfies
\begin{enumerate}[(i)]
\item for every \(n \in \N\), there exists \(C_n > 0\) such that, for every \(\zeta \in C_0^\infty(\overline{\Omega})\),
\[
\abs*{\int_{\Omega} \zeta \d\mu_n} \le C_n \norm{\Delta \zeta}_{L^p(\Omega)};
\]
\item \((\mu_n)_{n \in \N}\) converges strongly to \(\mu\) in \(\cM(\Omega)\).
\end{enumerate}
\end{proposition}

By the Riesz representation theorem, the functional inequality in \Cref{prop:strongApproximation} amounts to saying that the solution \(v_n\) of the Dirichlet problem
\[
\left\{
\begin{alignedat}{2}
-\Delta v_n &= \mu_n && \quad \text{in \(\Omega\)}, \\
v_n &= 0 && \quad \text{on \(\partial\Omega\)},
\end{alignedat}
\right.
\]
belongs to \(L^{p'}(\Omega)\) and satisfies \(\norm{v_n}_{L^{p'}(\Omega)} \le C_n\), where \(p'\) is the conjugate exponent of \(p \ge 1\).

To prove \Cref{prop:strongApproximation}, we need the following lemma:

\begin{lemma}
\label{lem:epsilonStrongApproximation}
Let \(1 \le p < \infty\) and \(\mu \in \cM(\Omega)\) be a nonnegative measure. If \(\mu\) is diffuse with respect to the \(\cl{p}\) capacity, then for every \(\varepsilon > 0\) there exists \(\nu \in \cM(\Omega)\) which satisfies
\begin{enumerate}[(i)]
\item there exists \(C > 0\) such that, for every \(\zeta \in C_0^\infty(\overline{\Omega})\),
\[
\abs*{\int_{\Omega} \zeta \d\nu} \le C \norm{\Delta \zeta}_{L^p(\Omega)};
\]
\item \(0 \le \nu \le \mu\) in \(\Omega\) and \(\norm{\mu - \nu}_{\cM(\Omega)} \le \varepsilon\).
\end{enumerate}
\end{lemma}

We rely on the straightforward weak capacitary inequality
\begin{equation}
\label{eq:weakCapacitaryInequality}
\capt_{\cl{p}}(\{\abs{\zeta} \ge 1\}; \Omega){}
\le 2^{p} \norm{\Delta\zeta}_{L^{p}(\Omega)}^{p},
\end{equation}
for every \(\zeta \in C_{0}^{\infty}(\overline\Omega)\). To see why this is true, we may assume that \(\zeta \ne 0\) and take \(h \in C^{\infty}(\overline\Omega)\) such that \(h > \abs{\Delta\zeta}\) in \(\Omega\) and 
\[{}
\norm{h}_{L^{p}(\Omega)}
\le 2 \norm{\Delta\zeta}_{L^{p}(\Omega)}.
\]
By the classical weak maximum principle, we can estimate the capacity of the set \(\{\abs{\zeta} \ge 1\}\) using the nonnegative solution of the Dirichlet problem
\[{}
\left\{
\begin{alignedat}{2}
-\Delta v &= h && \quad \text{in \(\Omega\)}, \\
v &= 0 && \quad \text{on \(\partial\Omega\)},
\end{alignedat}
\right.
\]
and then inequality~\eqref{eq:weakCapacitaryInequality} follows from the definition of the capacity.

\begin{proof}[Proof of \Cref{lem:epsilonStrongApproximation}]
Let \(\Phi : C_0^\infty(\overline{\Omega}) \rightarrow \R\) be the functional defined for \(\zeta \in C_0^\infty(\overline{\Omega})\) by
\[
\Phi(\zeta) = \int_{\Omega} \zeta^+ \d\mu;
\]
we equip \(C_0^\infty(\overline{\Omega})\) with the strong topology induced by the norm
\[
\zeta \longmapsto \norm{\Delta \zeta}_{L^p(\Omega)}.
\]

\begin{claim}
The functional \(\Phi\) is convex and lower semicontinuous.
\end{claim}

\begin{proof}[Proof of the claim]
The convexity of \(\Phi\) follows from the convexity of the real function \(t \in \R \mapsto t^+\). For the lower semicontinuity, let \((\zeta_n)_{n \in \N}\) be a sequence of functions in \(C_0^\infty(\overline{\Omega})\) converging to \(\zeta \in C_0^\infty(\overline{\Omega})\).{}
Applying the weak capacitary estimate~\eqref{eq:weakCapacitaryInequality} to \((\zeta_{n} - \zeta)/\epsilon\) with \(\epsilon > 0\), we deduce that \((\zeta_{n})_{n \in \N}\) converges to \(\zeta\) in capacity.
By absolute continuity of \(\mu\) with respect to the \(\cl{p}\) capacity, the convergence also holds in measure. By Fatou's lemma, we thus have
\[
\Phi(\zeta) = \int_{\Omega} \zeta^+ \d\mu \le \liminf_{n \rightarrow \infty} \int_{\Omega} \zeta_n^+ \d\mu = \liminf_{n \rightarrow \infty} \Phi(\zeta_n),
\]
and this proves the claim.
\end{proof}

We proceed with the proof of the lemma. Let \(0 < \delta < \varepsilon\) and take a compact set \(K \subset \Omega\) such that
\[
\mu(\Omega \setminus K) \le \varepsilon - \delta.
\]
By the claim and the geometric form of the Hahn--Banach theorem \cite{Brezis:2011}*{Theorem~1.11}, the functional \(\Phi\) is the supremum of a family of continuous linear functionals in \(C_0^\infty(\overline{\Omega})\). Hence, given a nonnegative function \(\psi \in C_0^\infty(\overline{\Omega})\) with \(\psi \ge 1\) in \(K\), there exists a continuous linear functional \(F : C_0^\infty(\overline{\Omega}) \rightarrow \R\) such that, for every \(\zeta \in C_0^\infty(\overline{\Omega})\),
\begin{equation}
\label{eq:functionalComparison}
F(\zeta) \le \Phi(\zeta),
\end{equation}
and
\begin{equation}
\label{eq:psiFunctionalComparison}
\Phi(\psi) \le F(\psi) + \delta.
\end{equation}
In particular,
\[
F(\zeta) \le \mu(\Omega) \norm{\zeta}_{L^{\infty}(\Omega)},
\]
for every \(\zeta \in C_{0}^{\infty}(\overline\Omega)\). Thus, by the Riesz representation theorem, there exists \(\nu \in \cM(\Omega)\) such that, for every \(\zeta \in C_0^\infty(\overline{\Omega})\),
\[
F(\zeta) = \int_{\Omega} \zeta \d\nu.
\]
Given a nonnegative function \(\zeta \in C_0^\infty(\overline{\Omega})\), by \eqref{eq:functionalComparison} we have \(F(- \zeta) \le 0\). Thus,
\[
0 \le \int_{\Omega} \zeta \d\nu \le \int_{\Omega} \zeta \d\mu,
\]
which implies that \(0 \le \nu \le \mu\) in \(\Omega\). Using \eqref{eq:psiFunctionalComparison}, we thus have
\[
\begin{aligned}
\norm{\mu - \nu}_{\cM(\Omega)} &= (\mu - \nu)(K) + (\mu - \nu)(\Omega \setminus K) \\
& \le \int_{\Omega} \psi \d(\mu - \nu) + \mu(\Omega \setminus K)\\
& = \Phi(\psi) - F(\psi) + \mu(\Omega \setminus K)
 \le \delta + (\varepsilon - \delta) = \varepsilon.
\end{aligned}
\]
Since \(F\) is a continuous linear functional in \(C_{0}^{\infty}(\overline\Omega)\), Assertion~{\itshape (i)} is satisfied and the proof of the lemma is complete.
\end{proof}

In the proof of \Cref{prop:strongApproximation}, we construct the sequence \((\mu_n)_{n \in \N}\) inductively based on \Cref{lem:epsilonStrongApproximation}.

\begin{proof}[Proof of \Cref{prop:strongApproximation}]
Let \((\varepsilon_n)_{n \in \N}\) be a non-increasing sequence of positive numbers converging to \(0\) and \(\nu_0 \in \cM(\Omega)\) be a measure satisfying the conclusion of \Cref{lem:epsilonStrongApproximation} with \(\varepsilon = \varepsilon_0\). Given \(n \in \N \setminus \set{0}\), assume that \((\nu_k)_{k \in \set{0,\dots,n-1}}\) is a family of nonnegative measures in \(\cM(\Omega)\) such that
\[
0 \le \sum_{k=0}^{n-1} \nu_k \le \mu
\]
and, for each \(k \in \set{0,\dots,n-1}\), \(\nu_k\) satisfies the functional inequality in Assertion~{\itshape (i)}. Applying \Cref{lem:epsilonStrongApproximation} to the measure \(\mu - \sum\limits_{k=0}^{n-1} \nu_k\) and \(\varepsilon = \varepsilon_n\), there exists a nonnegative measure \(\nu_n \in \cM(\Omega)\) satisfying the functional inequality in Assertion~{\itshape (i)} such that
\[
0 \le \nu_n \le \mu - \sum_{k=0}^{n-1} \nu_k \quad \text{and} \quad \norm[\bigg]{\mu - \sum_{k=0}^{n-1} \nu_k - \nu_n}_{\cM(\Omega)} \le \varepsilon_n.
\]
For each \(n \in \N\), define \(\mu_n = \sum\limits_{k=0}^n \nu_k\). Such a sequence \((\mu_n)_{n \in \N}\) satisfies Assertions~{\itshape (i)}~and~{\itshape (ii)} but the measures need not be compactly supported in \(\Omega\). To this end, for each \(n \in \N\) we define
\[
\Omega_n = \set*{x \in \Omega : \dist(x,\partial\Omega) > \varepsilon_n}.
\]
Observe that \(\Omega_n \Subset \Omega\) and \(\Omega = \bigcup\limits_{n=0}^\infty \Omega_n\). By the monotone set property, for every \(i \in \N\), we have
\[
\lim_{n \rightarrow \infty} \mu_i(\Omega \setminus \Omega_n) = \mu_i(\Omega \setminus \Omega) = 0.
\]
By the triangle inequality, we also have
\[
\norm{\mu - \mu_i \lfloor_{\Omega_n}}_{\cM(\Omega)} 
\le \norm{\mu - \mu_i}_{\cM(\Omega)} + \norm{\mu_i - \mu_i \lfloor_{\Omega_n}}_{\cM(\Omega)}
\le \varepsilon_i + \mu_i(\Omega \setminus \Omega_n).
\]
Take an increasing sequence of indices \((n_i)_{i \in \N}\) such that, for every \(i \in \N\),
\[
\mu_i(\Omega \setminus \Omega_{n_i}) \le 1/(i+1).
\]
Then, the sequence \((\mu_{n_i} \lfloor_{\Omega_{n_i}})_{i \in \N}\) has the required properties.
\end{proof}

\section{Choice of a minimizing sequence for the capacity}

To show that a measure for which the Dirichlet problem~\eqref{eq:dirichletProblem} has a solution for every nonnegative potential in \(L^p(\Omega)\) is diffuse, we rely on a particular choice of a minimizing sequence for the capacity using a cut-off and truncation argument:

\begin{proposition}
\label{prop:minimizingSequence}
Let \(1 \le p < \infty\) and \(K \subset \Omega\) be a compact set such that \(\capt_{\cl{p}}(K;\Omega) = 0\). Then there exists a sequence \((\varphi_n)_{n \in \N}\) of nonnegative functions in \(C_c^\infty(\Omega)\) such that
\begin{enumerate}[(i)]
\item \((\varphi_n)_{n \in \N}\) converges pointwise to the characteristic function \(\chi_K\);
\item \((\varphi_n)_{n \in \N}\) is bounded in \(L^\infty(\Omega)\);
\item \((\Delta \varphi_n)_{n \in \N}\) converges to \(0\) in \(L^p(\Omega)\).
\end{enumerate}
\end{proposition}

The proof of \Cref{prop:minimizingSequence} relies on two lemmas. The first one shows that minimizing functions of the capacity in the space \(C_0^\infty(\overline{\Omega})\) can be chosen to be compactly supported in \(\Omega\).

\begin{lemma}
\label{lem:minimizingSequenceCompactSupport}
Let \(1 \le p < \infty\) and \(K \subset \Omega\) be a compact set such that \(\capt_{\cl{p}}(K;\Omega) = 0\). Then there exists a sequence \((\varphi_n)_{n \in \N}\) of nonnegative functions in \(C_c^\infty(\Omega)\) such that
\begin{enumerate}[(i)]
\item for every \(n \in \N\), we have \(\varphi_n > 1\) in \(K\);
\item \((\Delta \varphi_n)_{n \in \N}\) converges to \(0\) in \(L^p(\Omega)\).
\end{enumerate}
\end{lemma}

\begin{proof}[Proof of \Cref{lem:minimizingSequenceCompactSupport}]
By the definition of the capacity of \(K\), there exists a sequence \((\zeta_n)_{n \in \N}\) of nonnegative functions in \(C_0^\infty(\overline{\Omega})\) such that, for every \(n \in \N\), \(\zeta_n > 1\) in \(K\), and the sequence \((\Delta \zeta_n)_{n \in \N}\) converges to \(0\) in \(L^p(\Omega)\). Take a fixed nonnegative function \(\phi \in C_c^\infty(\Omega)\) such that \(\phi = 1\) in \(K\). For each \(n \in \N\), define \(\varphi_n = \zeta_n \phi\). On the one hand, we have
\[
\norm{\Delta \varphi_n}_{L^p(\Omega)} \le \C \bracks*{\norm{\zeta_n}_{W^{1,p}(\Omega)} + \norm{\Delta \zeta_n}_{L^p(\Omega)}}.
\]
On the other hand, there exists a constant \(C > 0\) such that, for every \(\zeta \in C_0^\infty(\overline{\Omega})\), the following estimate holds:
\begin{equation}
\label{eq:ellipticEstimateStampacchia}
\norm{\zeta}_{W^{1,p}(\Omega)} \le C \norm{\Delta \zeta}_{L^p(\Omega)}.
\end{equation}
In the case \(p = 1\), this is proved by Littman, Stampacchia and Weinberger~\cite{Littman_Stampacchia_Weinberger:1963}*{Theorem~5.1}; see also \cite{Ponce:2016}*{Proposition~5.1}. In the case \(p > 1\), this is a consequence of \(L^p\) estimates due to Cald\'eron and Zygmund~\cite{Gilbarg_Trudinger:1983}*{Theorem~9.15 and Lemma~9.17}. Applying this estimate to the functions \(\zeta_n\), we obtain
\[
\norm{\Delta \varphi_n}_{L^p(\Omega)} \le \C \norm{\Delta \zeta_n}_{L^p(\Omega)}.
\]
The sequence \((\varphi_n)_{n \in \N}\) has the required properties.
\end{proof}

The second lemma involved in the proof of \Cref{prop:minimizingSequence} is a property of composition with nonnegative test functions due to Maz'ya~\cite{Maz'ya:1973}. We present a proof for convenience.

\begin{lemma}
\label{lem:compositionMazya}
Let \(1 \le p < \infty\) and \(H : \left[ 0,\infty \right) \rightarrow \R\) be a smooth function such that \(H(0) = 0\) and \(H(s) = 1\), for every \(s \ge 1\). Then for every nonnegative function \(\varphi \in C_c^\infty(\R^N)\) we have
\[
\norm{\Delta H(\varphi)}_{L^p(\R^N)} \le C \norm{\Delta \varphi}_{L^p(\R^N)},
\]
for some constant \(C > 0\) depending on \(N\), \(p\) and \(H\).
\end{lemma}

\resetconstant
\begin{proof}[Proof of \Cref{lem:compositionMazya}]
Since \(H''(\varphi)\) is supported in \(\set{\varphi \le 1}\), we have
\[
\abs{\Delta H(\varphi)} \le \C \bracks*{\abs{\Delta \varphi} + \chi_{\set{\varphi \le 1}} \abs{\nabla \varphi}^2}.
\]
By the triangle inequality, we obtain the estimate
\[
\norm{\Delta H(\varphi)}_{L^p(\R^N)} \le \Cs \bracks*{\norm{\Delta \varphi}_{L^p(\R^N)} + \norm{\nabla \varphi}_{L^{2p}(\set{\varphi \le 1})}^2}.
\]
The remaining of the proof consists in showing that
\begin{equation}
\label{eq:estimateGradientLaplacian}
\norm{\nabla \varphi}_{L^{2p}(\set{\varphi \le 1})}^2 \le {\C} \norm{\Delta \varphi}_{L^p(\R^N)}.
\end{equation}
First of all, observe that
\[
\Div \bracks*{\frac{\abs{\nabla \varphi}^{2p-2} \nabla \varphi}{(1+\varphi)^{2p-1}}} = -(2p-1) \frac{\abs{\nabla \varphi}^{2p}}{(1+\varphi)^{2p}} + \frac{\Div(\abs{\nabla \varphi}^{2p-2} \nabla \varphi)}{(1+\varphi)^{2p-1}}.
\]
By the Divergence theorem, we obtain
\[
\int_{\R^N} \frac{\abs{\nabla \varphi}^{2p}}{(1+\varphi)^{2p}} = \frac{1}{2p-1} \int_{\R^N} \frac{\Div(\abs{\nabla \varphi}^{2p-2} \nabla \varphi)}{(1+\varphi)^{2p-1}}.
\]
In the case \(p = 1\), we immediately deduce \eqref{eq:estimateGradientLaplacian} from the nonnegativity of \(\varphi\):
\[
\frac{1}{4} \int_{\set{\varphi \le 1}} \abs{\nabla \varphi}^2 \le \int_{\R^N} \abs{\Delta \varphi}.
\]
Let us thus assume that \(p > 1\). From the integral identity above, we have
\[
\int_{\R^N} \frac{\abs{\nabla \varphi}^{2p}}{(1+\varphi)^{2p}} \le \C \int_{\R^N} \frac{\abs{D^2 \varphi} \abs{\nabla \varphi}^{2p-2}}{(1+\varphi)^{2p-1}}.
\]
By the H\"older inequality, we then get
\[
\int_{\R^N} \frac{\abs{\nabla \varphi}^{2p}}{(1+\varphi)^{2p}} \le \Cs \paren[\bigg]{\int_{\R^N} \frac{\abs{D^2 \varphi}^p}{(1+\varphi)^p}}^{\frac{1}{p}} \paren[\bigg]{\int_{\R^N} \frac{\abs{\nabla \varphi}^{2p}}{(1+\varphi)^{2p}}}^{\frac{p-1}{p}}.
\]
Since \(\varphi\) is nonnegative, this yields the estimate
\[
\frac{1}{4} \paren[\bigg]{\int_{\set{\varphi \le 1}} \abs{\nabla \varphi}^{2p}}^{\frac{1}{p}} \le \Cs \paren[\bigg]{\int_{\R^N} \abs{D^2 \varphi}^p}^{\frac{1}{p}}.
\]
On the other hand, the Cald\'eron--Zygmund inequality \cite{Gilbarg_Trudinger:1983}*{Corollary~9.10} implies the existence of a constant \(C > 0\) independent of \(\varphi\) such that
\[
\norm{D^2 \varphi}_{L^p(\R^N)} \le C \norm{\Delta \varphi}_{L^p(\R^N)}.
\]
Combining the last two inequalities, we obtain \eqref{eq:estimateGradientLaplacian} and the lemma follows.
\end{proof}


We now turn to the

\resetconstant
\begin{proof}[Proof of \Cref{prop:minimizingSequence}]
Let \((\varphi_n)_{n \in \N}\) be a sequence of nonnegative functions in \(C_c^\infty(\Omega)\) satisfying the conclusion of \Cref{lem:minimizingSequenceCompactSupport}. Take a smooth function \(H : \left[{} 0, \infty \right){} \rightarrow \R\) satisfying the assumptions of \Cref{lem:compositionMazya} and let \(\psi \in C_c^\infty(\Omega)\) be given. For every \(n \in \N\), we have
\[
\norm[\big]{\Delta (\psi H(\varphi_n))}_{L^p(\Omega)} \le \C \bracks*{\norm[\big]{H(\varphi_n)}_{W^{1,p}(\Omega)} + \norm[\big]{\Delta H(\varphi_n)}_{L^p(\Omega)}}.
\]
Applying the estimate \eqref{eq:ellipticEstimateStampacchia} to \(H(\varphi_n)\) and then using \Cref{lem:compositionMazya}, we obtain
\[
\norm[\big]{\Delta (\psi H(\varphi_n))}_{L^p(\Omega)} \le \C \norm{\Delta \varphi_n}_{L^p(\Omega)}.
\]
Letting \(n\) tend to infinity, we get
\begin{equation}
\label{eq:limitLaplacian}
\lim_{n \rightarrow \infty} \norm{\Delta (\psi H(\varphi_n))}_{L^p(\Omega)} = 0.
\end{equation}

To obtain the conclusion, we use Cantor's diagonal argument. For this purpose, take a non-increasing sequence \((\omega_k)_{k \in \N}\) of open subsets of \(\Omega\) containing \(K\) such that
\[\bigcap\limits_{k=0}^\infty \omega_k = K.\]
For every \(k \in \N\), choose a function \(\psi_k \in C_c^\infty(\omega_k)\) such that \(\psi_k = 1\) in \(K\). Given a sequence \((\varepsilon_k)_{k \in \N}\) of positive numbers converging to \(0\), we deduce from \eqref{eq:limitLaplacian} that there exists an increasing sequence \((n_k)_{k \in \N}\) of indices such that, for every \(k \in \N\),
\[
\norm[\big]{\Delta (\psi_k H(\varphi_{n_k}))}_{L^p(\Omega)} \le \varepsilon_k.
\]
The sequence \((\psi_k H(\varphi_{n_k}))_{k \in \N}\) has the required properties.
\end{proof}

\section{Proof of \texorpdfstring{\Cref{thm:characterization}}{Theorem \ref{thm:characterization}}}

Suppose that \(\mu\) is diffuse with respect to \(\capt_{\cl{p}}\). Since the equation is linear and \(\mu\) can be decomposed as a difference of nonnegative diffuse measures, we may assume from the beginning that \(\mu\) is nonnegative. Indeed, by the Jordan decomposition theorem there exist two nonnegative measures \(\mu^+\) and \(\mu^-\) in \(\cM(\Omega)\) which are mutually singular and such that \(\mu = \mu^+ - \mu^-\). Using the inner regularity of the measures, one shows that they are both diffuse with respect to \(\capt_{\cl{p}}\). Thus, assuming \(\mu\) to be nonnegative, let \((\mu_i)_{i \in \N}\) be a sequence of nonnegative measures in \(\cM(\Omega)\) satisfying the conclusion of \Cref{prop:strongApproximation}. For each \(i \in \N\), denote by \(v_i\) the solution of the Dirichlet problem
\[
\left\{
\begin{alignedat}{2}
-\Delta v_i &= \mu_i && \quad \text{in \(\Omega\)}, \\
v_i &= 0 && \quad \text{on \(\partial\Omega\)}.
\end{alignedat}
\right.
\]
It follows from the functional inequality satisfied by \(\mu_i\) that, for every \(f \in C^\infty(\overline{\Omega})\),
\[
\abs*{\int_{\Omega} f v_i} \le C_i \norm{f}_{L^p(\Omega)}.
\]
By the Riesz representation theorem, this implies that \(v_i \in L^{p'}(\Omega)\), whence
\[
V v_i \in L^1(\Omega).
\]
Since \(\mu_i\) is nonnegative, by the weak maximum principle, we have
\[
v_i \ge 0 \quad \text{almost everywhere in \(\Omega\)}.
\]
Then, by the nonnegativity of \(V\), we also have
\[
Vv_i \ge 0 \quad \text{almost everywhere in \(\Omega\)}.
\]
The function \(v_i\) is thus a supersolution of the Dirichlet problem~\eqref{eq:dirichletProblem} with datum \(\mu_i\). Applying the method of sub and supersolutions for the Schr\"odinger operator \cite{Ponce:2016}*{Proposition~22.7} with subsolution \(0\) and supersolution \(v_i\), we deduce that the Dirichlet problem~\eqref{eq:dirichletProblem} with datum \(\mu_i\) has a nonnegative solution \(u_i\). Due to the linearity of the equation, for every \(i,j \in \N\), the function \(u_i - u_j\) is a solution of the Dirichlet problem~\eqref{eq:dirichletProblem} with datum \(\mu_i - \mu_j\). By the absorption estimate \eqref{eq:absorptionEstimate} (cf.~\cite{Ponce:2016}*{Proposition~21.5}), we then obtain
\[
\norm{Vu_i - Vu_j}_{L^1(\Omega)} \le \norm{\mu_i - \mu_j}_{\cM(\Omega)}.
\]
Since the sequence of measures \((\mu_i)_{i \in \N}\)  converges strongly in \(\cM(\Omega)\), the inequality above implies that \((Vu_i)_{i \in \N}\) is a Cauchy sequence in \(L^1(\Omega)\). It then follows from the \(L^1\) elliptic estimate (cf.~\eqref{eq:ellipticEstimateStampacchia}) and the triangle inequality that
\[
\norm{u_i - u_j}_{L^1(\Omega)} \le \norm{\mu_i - \mu_j}_{\cM(\Omega)} + \norm{Vu_i - Vu_j}_{L^1(\Omega)}.
\]
Hence, \((u_i)_{i \in \N}\) is also a Cauchy sequence in \(L^1(\Omega)\) and thus converges in \(L^1(\Omega)\) to some function \(u\). This implies that the sequence \((Vu_i)_{i \in \N}\) converges in \(L^1(\Omega)\) to the function \(Vu\). By the Dominated convergence theorem, we deduce that
\[
\int_{\Omega} u(-\Delta \zeta + V \zeta) = \int_{\Omega} \zeta \d\mu,
\]
for every \(\zeta \in C_0^\infty(\overline{\Omega})\). The function \(u\) is therefore a solution of the Dirichlet problem~\eqref{eq:dirichletProblem} with datum \(\mu\).

For the converse implication, assume that the Dirichlet problem~\eqref{eq:dirichletProblem} with datum \(\mu\) has a solution for every nonnegative function \(V \in L^p(\Omega)\). Let \(K \subset \Omega\) be a compact set such that \(\capt_{\cl{p}}(K;\Omega) = 0\) and let \((\varphi_n)_{n \in \N}\) be a sequence of nonnegative functions in \(C_c^\infty(\Omega)\) satisfying the conclusion of \Cref{prop:minimizingSequence}. Since the sequence \((\Delta \varphi_n)_{n \in \N}\) converges to \(0\) in \(L^p(\Omega)\), we deduce from the partial converse of the Dominated convergence theorem \cite{Willem:2013}*{Proposition 4.2.10} that there exist a subsequence \((\varphi_{n_k})_{k \in \N}\) and a function \(V \in L^p(\Omega)\) such that
\begin{enumerate}[(a)]
\item for every \(k \in \N\), \(\abs{\Delta \varphi_{n_k}} \le V\) almost everywhere in \(\Omega\);
\item \((\Delta \varphi_{n_k})_{k \in \N}\) converges almost everywhere to \(0\) in \(\Omega\).
\end{enumerate}
Let \(u\) be the solution of the Dirichlet problem~\eqref{eq:dirichletProblem} with potential \(V\) and density \(\mu\). For every \(k \in \N\),
\[
\abs{u \Delta \varphi_{n_k}} \le \abs{Vu} \in L^1(\Omega).
\]
Then, by the Dominated convergence theorem,
\[
\lim_{k \rightarrow \infty} \int_{\Omega} u \Delta \varphi_{n_k} = 0.
\]
Since the sequence \((\varphi_{n_k})_{k \in \N}\) is bounded in \(L^\infty(\Omega)\) and converges pointwise to the characteristic function \(\chi_K\), another application of the Dominated convergence theorem yields
\[
\lim_{k \rightarrow \infty} \int_{\Omega} \varphi_{n_k} \d\mu = \mu(K) \quad \text{and} \quad \lim_{k \rightarrow \infty} \int_{\Omega} Vu \varphi_{n_k} = 0.
\]
Combining the above limits, we obtain
\[
\mu(K) = \lim_{k \rightarrow \infty} \int_{\Omega} \varphi_{n_k} \d\mu = \lim_{k \rightarrow \infty} \int_{\Omega} u(-\Delta \varphi_{n_k} + V \varphi_{n_k}) = 0.
\]
Since \(K\) is arbitrary, we conclude that \(\mu\) is diffuse with respect to \(\capt_{\cl{p}}\). The proof of the theorem is complete. \qed

\begin{remark}
\label{rem:remark}
In the paper of V\'eron and Yarur~\cite{Veron_Yarur:2012}, the authors investigate a counterpart of the Dirichlet problem~\eqref{eq:dirichletProblem} for the trace problem. In their case, the measure lies on the boundary instead of in the interior of the domain and they also assume that \(V\) belongs to \(L\loc^\infty(\Omega)\). In our case, the Dirichlet problem~\eqref{eq:dirichletProblem} has a solution for every diffuse measure even if the potential \(V\) merely belongs to \(L\loc^p(\Omega)\), and the same observation applies to \Cref{thm:existence} when \(V \in H\loc^1(\Omega)\).

Rather than deducing this fact from our \Cref{thm:characterization,thm:existence}, it is more convenient to implement directly the tools developed here. We explain the argument for \(V \in L\loc^p(\Omega)\). To this end, we combine the strategies in \cite{Baras_Pierre:1984}*{Lemme~3.2} for \(p > 1\) and \cite{Orsina_Ponce:2017}*{Proposition~3.1} for \(p = 1\). We first assume that \(\mu \in \cM(\Omega)\) has compact support in \(\Omega\) and satisfies
\begin{equation}
\label{eq:measureRemark}
\abs*{\int_{\Omega} \zeta \d\mu} \le C \norm{\Delta \zeta}_{L^p(\Omega)},
\end{equation}
for every \(\zeta \in C_0^\infty(\overline{\Omega})\). For a sequence of mollifiers \((\rho_n)_{n \in \N}\) in \(C_c^\infty(\R^N)\) supported in a small neighborhood of \(0\), one still has
\[
\abs*{\int_{\Omega} \zeta \, \rho_n * \mu} \le C' \norm{\Delta \zeta}_{L^p(\Omega)}.
\]
Hence, the solution \(u_n\) of the Dirichlet problem~\eqref{eq:dirichletProblem} with density \(\rho_n * \mu\) satisfies the uniform bound \(\norm{u_n}_{L^{p'}(\Omega)} \le C'\) and also the absorption estimate
\[
\norm{Vu_n}_{L^1(\Omega)} \le \norm{\rho_n * \mu}_{L^1(\Omega)} \le \norm{\mu}_{\cM(\Omega)}.
\]
By compactness, the sequence \((u_n)_{n \in \N}\) converges strongly in \(L^1(\Omega)\) and weakly in \(L^{p'}(\Omega)\) to some function \(u\), and so both estimates are also satisfied by \(u\). Next, for every \(\varepsilon > 0\), one writes
\[
\int_{\Omega} V u_n \zeta = \int_{\set{\abs{\zeta} \ge \varepsilon}} V u_n \zeta + \int_{\set{\abs{\zeta} < \varepsilon}} V u_n \zeta.
\]
Note that \(V \zeta \chi_{\set{\abs{\zeta} \ge \varepsilon}} \in L^p(\Omega)\), while the last integral is uniformly bounded in absolute value by \(\varepsilon \norm{\mu}_{\cM(\Omega)}\). Thus, as \(n\) tends to infinity and then \(\varepsilon\) tends to zero, one deduces that
\[
\lim_{n \rightarrow \infty} \int_{\Omega} V u_n \zeta = \int_{\Omega} V u \zeta.
\]
Hence, \(u\) satisfies the Dirichlet problem~\eqref{eq:dirichletProblem} under assumption~\eqref{eq:measureRemark}. For an arbitrary measure \(\mu \in \cM(\Omega)\) which is diffuse with respect to the \(\cl{p}\) capacity, one proceeds along the lines of the proof of the direct implication of \Cref{thm:characterization} by strong approximation of \(\mu\) in \(\cM(\Omega)\) in terms of measures with compact support that satisfy \eqref{eq:measureRemark}.
\end{remark}

\section{A geometric interpretation of \texorpdfstring{\Cref{thm:existence}}{Theorem \ref{thm:existence}}}
\label{sec:geometricInterpretation}

In this section, we provide a geometric interpretation of \Cref{thm:existence} which involves the Hausdorff content \(\haus_\infty^{N-2}\), defined for every compact set \(K \subset \R^N\) by
\[
\haus_\infty^{N-2}(K) = \inf \set[\Bigg]{\sum_{i=0}^n \omega_{N-2} r_i^{N-2} : \text{\(K \subset \bigcup_{i=0}^n B(x_i;r_i)\) and \(0 < r_i < \infty\)}},
\]
where \(\omega_{N-2}\) is the volume of the unit ball in \(\R^{N-2}\). The Hausdorff content is always finite and vanishes on the same compact sets as the Hausdorff measure \(\haus^{N-2}\).

To reach our goal, we rely on the following second-order capacity
\[
\capt_{\cd}(K) = \inf \set[\Big]{\norm{D^2 \varphi}_{L^1(\R^N)} : \text{\(\varphi \in C_c^\infty(\R^N)\) is nonnegative and \(\varphi > 1\) in \(K\)}}.
\]
The connection between \(\haus_\infty^{N-2}\) and \(\capt_{\ch}\) through this capacity can be summarized as follows:

\begin{proposition}
\label{prop:equivalence}
Suppose that \(N \ge 3\). Then
\[
\capt_{\cd} \sim \capt_{\ch} \sim \haus_\infty^{N-2}
\]
on every compact subset of\/ \(\R^N\).
\end{proposition}

For two capacities \(\capt\) and \(\capt'\), by \(\capt \sim \capt'\) we mean that there exist positive constants \(C_1\) and \(C_2\) such that
\[
C_1 \capt \le \capt' \le C_2 \capt.
\]

As a straightforward consequence of \Cref{prop:equivalence}, we have that measures which are diffuse with respect to the \(\ch\) capacity cannot charge compact sets of zero \(\haus^{N-2}\) measure. \Cref{prop:equivalence} can be deduced from the work of Adams~\cite{Adams:1988}. His proof that yields the equivalence \(\capt_{\ch} \sim \haus_\infty^{N-2}\) is based on the duality between the Hardy space \(H^1\) and the space \(\BMO\) of functions of bounded mean oscillation. We provide a short argument which relies on the boundedness of the Riesz transform in the Hardy space \(H^1\).

We recall that the Riesz transform of a function \(\varphi \in C_c^\infty(\R^N)\) is the vector-valued function \(R \varphi : \R^N \rightarrow \R^N\) defined for \(x \in \R^N\) by
\[
\RT \varphi (x) = - C_N \int_{\R^N} \frac{\nabla \varphi(y)}{\abs{x - y}^{N-1}} \d{y},
\]
where \(C_N\) is a positive constant depending on \(N\). In particular, \(\RT \varphi \in L^\infty(\R^N)\). The real Hardy space is the vector subspace of \(L^1(\R^N)\) given by
\[
H^1(\R^N) = \set*{u \in L^1(\R^N) : \RT u \in L^1(\R^N;\R^N)},
\]
where \(\RT u\) is defined in the sense of distributions.

\resetconstant

\begin{proof}[Proof of \Cref{prop:equivalence}]
By translation and scaling arguments, for every \(x \in \R^N\) and every \(r > 0\), one has
\begin{equation}
\label{eq:secondOrderCapacityClosedBall}
\capt_{\cd}(B[x;r]) = r^{N-2} \capt_{\cd}(B[0;1]).
\end{equation}
The finite subadditivity of the \(\cd\) capacity then implies that
\[
\capt_{\cd} \le \C \haus_\infty^{N-2}.
\]
The same argument yields the estimate
\begin{equation}
\label{eq:hardyHausdorffEstimate}
\capt_{\ch} \le \C \haus_\infty^{N-2}.
\end{equation}
Indeed, given \(a \in \R^N\) and \(r > 0\), for every \(\varphi \in C_c^\infty(\R^N)\), one observes that
\[
\norm[\bigg]{\Delta \varphi \paren[\bigg]{\frac{\cdot - a}{r}}}_{H^1(\R^N)} = r^{N-2} \norm{\Delta \varphi}_{H^1(\R^N)},
\]
and this identity yields the counterpart of \eqref{eq:secondOrderCapacityClosedBall} for \(\capt_{\ch}\). Thus, one also has \eqref{eq:hardyHausdorffEstimate}.

The Riesz transform maps continuously functions in \(H^1(\R^N)\) into \(H^1(\R^N)\) \cite{Stein:1993}*{Chapter~III, Theorem~4}. More precisely,
\[
\norm{\RT u}_{H^1(\R^N)} \le \C \norm{u}_{H^1(\R^N)},
\]
for every \(u \in H^1(\R^N)\). This yields the estimate
\[
\capt_{\cd} \le \C \capt_{\ch}.
\]
Indeed, for every \(\varphi \in C_c^\infty(\R^N)\), one has
\begin{equation}
\label{eq:secondOrderLaplacian}
\norm[\big]{D^2 \varphi}_{L^1(\R^N)} 
\le \C \norm[\big]{\RT (\RT \Delta \varphi)}_{L^1(\R^N)} 
\le \Cs \norm[\big]{\RT \Delta \varphi}_{H^1(\R^N)} \le \C \norm{\Delta \varphi}_{H^1(\R^N)},
\end{equation}
where the first inequality follows from the Fourier characterization of the Riesz transform \cite{Stein:1970}*{p.~59}.

The estimate
\[
\haus_\infty^{N-2} \le \C \capt_{\cd}
\]
is a consequence of the following second-order counterpart of Gustin's boxing inequality:
\begin{equation}
\label{eq:boxingInequality}
\haus_\infty^{N-2} \paren*{\set*{\abs{\varphi} \ge 1}} \le \C \norm{D^2 \varphi}_{L^1(\R^N)},
\end{equation}
for every \(\varphi \in C_c^\infty(\R^N)\). This inequality is based on two ingredients. The first one is a trace inequality due to Maz'ya~\citelist{\cite{Maz'ya:1979} \cite{Maz'ya:2011}*{Section~1.4.3}}: if \(\mu\) is a nonnegative finite Borel measure such that \(\mu \le \haus_\infty^{N-2}\), then
\[
\int_{\R^N} \abs{\varphi} \d\mu \le \C \int_{\R^N} \abs{D^2 \varphi}.
\]
The second one is Frostman's lemma, which provides one with a nonnegative finite Borel measure \(\mu\) supported by \(\set{\abs{\varphi} \ge 1}\) such that \(\mu \le \haus_\infty^{N-2}\) and
\[
\haus_\infty^{N-2}(\set{\abs{\varphi} \ge 1}) \le \C \, \mu(\set{\abs{\varphi} \ge 1}).
\]
Combining the last two inequalities, we obtain \eqref{eq:boxingInequality}. The proof of the proposition is complete.
\end{proof}

Before concluding this section, we observe that
\Cref{prop:equivalence} also holds for \(N = 2\). 
For example, the inequality \(\haus_{\infty}^0 \le \capt_{\cd}\) has a straightforward proof.
Indeed, on the one hand we have \(\haus_{\infty}^0 = 1\) on nonempty compact sets, while on the other hand \(\norm{\varphi}_{L^\infty(\R^2)} \le \norm{D^2 \varphi}_{L^1(\R^2)}\), and thus \(\capt_{\cd} \ge 1\). 
This yields the desired inequality; the rest of the argument is unchanged.

\section{Proof of \texorpdfstring{\Cref{thm:existence}}{Theorem \ref{thm:existence}}}

We first prove \Cref{thm:existence} for nonnegative, compactly supported measures which are explicitly controlled by the \(\ch\) capacity. The proof of \Cref{thm:existence} will then be carried out by strong approximation of \(\mu\) by measures of this type.

\begin{proposition}
\label{prop:existenceCompactSupport}
Let \(V \in H^1(\Omega)\) be a nonnegative function. If \(\mu \in \cM(\Omega)\) is a nonnegative measure with compact support such that \(\mu \le C \capt_{\ch}\), then the Dirichlet problem~\eqref{eq:dirichletProblem} has a nonnegative solution.
\end{proposition}

In dimension \(N \ge 3\), the proof of \Cref{prop:existenceCompactSupport} relies on the exponential integrability of the Newtonian potential generated by the measure \(\mu\), i.e. the function \(\NP\mu : \R^N \rightarrow \bracks{0,\infty}\) defined for \(x \in \R^N\) by
\[
\NP\mu(x) = \frac{1}{(N - 2) \sigma_N} \int_{\Omega} \frac{\d\mu(y)}{\abs{x - y}^{N- 2}},
\]
where \(\sigma_N\) denotes the surface measure of the unit sphere in \(\R^N\). Indeed, since \(\mu\) is nonnegative and satisfies (\Cref{prop:equivalence})
\begin{equation}
\label{eq:densityEstimate}
\mu \le C \haus_{\infty}^{N-2},
\end{equation}
one has \(\Exp^{\NP\mu / C} \in L^1(\Omega)\). This result is proved by Bartolucci, Leoni, Orsina and Ponce~\cite{Bartolucci_Leoni_Orsina_Ponce:2005}*{Theorem~2} and is the counterpart in higher dimensions of the Brezis--Merle inequality for \(N = 2\); see also \cite{Ponce:2016}*{Proposition~17.8}.

To make the connection with the Dirichlet problem we want to solve, observe that the Newtonian potential belongs to \(L^1(\Omega)\) and satisfies the Poisson equation \cite{Ponce:2016}*{Example~2.12}
\[
- \Delta \NP\mu = \mu \quad \text{in the sense of distributions in \(\Omega\)}.
\]
Besides, every nonnegative function \(V \in H^1(\Omega)\) is locally \(L \log L\) integrable, i.e. for every open set \(\omega \Subset \Omega\), one has
\[
\int_{\omega} V \log^+ V < \infty,
\]
where \(\log^+\) is the positive part of the logarithm function; see \citelist{\cite{Stein:1969}*{Theorem~3}}. In view of Young's inequality, the Newtonian potential is thus a supersolution of the equation
\[
- \Delta u + Vu = \mu \quad \text{in the sense of distributions in \(\Omega\)}.
\]

Since the datum \(\mu\) is nonnegative, to prove the existence of a solution of the Dirichlet problem~\eqref{eq:dirichletProblem} one may construct a supersolution --- not only to the equation, but also taking into account the boundary data --- and then apply the method of sub and supersolutions for the Schr\"odinger operator \cite{Ponce:2016}*{Proposition~22.7}. We observe that the Newtonian potential is indeed a supersolution of the Dirichlet problem~\eqref{eq:dirichletProblem} with datum \(\mu\). The formulation in this case involves test functions in \(C_0^\infty(\overline{\Omega})\) rather than in \(C_c^\infty(\Omega)\). More precisely, given \(V \in L^1(\Omega)\) and \(\nu \in \cM(\Omega)\), we say that \(u \in L^1(\Omega)\) is a supersolution of the Dirichlet problem~\eqref{eq:dirichletProblem} with datum \(\nu\) provided \(Vu \in L^1(\Omega)\) and
\[
- \Delta u + Vu \ge \nu \quad \text{in the sense of \((C_0^\infty(\overline{\Omega}))'\)},
\]
meaning that, for every nonnegative function \(\zeta \in C_0^\infty(\overline{\Omega})\),
\[
\int_{\Omega} u (- \Delta \zeta + V \zeta) \ge \int_{\Omega} \zeta \d\nu.
\]
This formulation encodes the boundary condition \(u \ge 0\) on \(\partial \Omega\). For example, if \(u \in L^1(\Omega)\), \(V \in L^1(\Omega)\) and \(\nu \in \cM(\Omega)\) are such that \(Vu \in L^1(\Omega)\) and
\[
- \Delta u + Vu \ge \nu \quad \text{in the sense of distributions in \(\Omega\)},
\]
and if \(u\) is nonnegative in \(\Omega\), then one has \cite{Ponce:2016}*{Lemma~17.6}
\[
- \Delta u + Vu \ge \nu \quad \text{in the sense of \((C_0^\infty(\overline{\Omega}))'\)}.
\]

The above discussion can be implemented as follows:

\resetconstant

\begin{proof}[Proof of \Cref{prop:existenceCompactSupport}]
Assume that \(N \ge 3\); the case \(N = 2\) will be explained afterwards. Since the Newtonian potential \(\NP\mu\) is nonnegative, it satisfies
\[
- \Delta \NP\mu \ge \mu \quad \text{in the sense of \((C_0^\infty(\overline{\Omega}))'\)}.
\]
It thus remains to show that \(V \NP\mu \in L^1(\Omega)\). We first observe that, by Young's inequality,
\begin{equation}
\label{eq:pointwiseInequality}
V \NP\mu \le \kappa \bracks[\Big]{\Exp^{\NP\mu / \kappa} + V \log^+ V} \quad \text{almost everywhere in \(\Omega\)},
\end{equation}
for every \(\kappa > 0\). On the other hand, by \Cref{prop:equivalence} there exists \(\Cl{cst:estimate} > 0\) such that
\[
\mu \le \Cs \haus_{\infty}^{N-2}.
\]
Thus, \(\Exp^{\NP\mu / \Cs} \in L^1(\Omega)\). For every \(\varepsilon > 0\), define the open set
\[
\Omega_{\varepsilon} = \set*{x \in \Omega : \dist(x, \partial \Omega) > \varepsilon}.
\]
Let \(\delta = \dist(\supp \mu, \partial \Omega)\). Since \(\NP\mu\) is harmonic in \(\R^N \setminus \overline{\Omega_{\delta}}\), we have \(\NP\mu \in L^{\infty}(\Omega \setminus \overline{\Omega_{\delta / 2}})\). Taking \(\kappa = \Cr{cst:estimate}\) in the pointwise inequality \eqref{eq:pointwiseInequality} above and integrating it over \(\Omega\), we obtain
\[
\int_{\Omega} V \NP\mu \le \Cr{cst:estimate} \int_{\Omega_{\delta / 2}} \bracks[\Big]{\Exp^{\NP\mu / \Cr{cst:estimate}} + V \log^+ V} + \C \norm{V}_{L^1(\Omega \setminus \overline{\Omega_{\delta / 2}})} < \infty,
\]
where the last inequality follows from the exponential integrability of \(\NP\mu\) and the local \(L \log L\) integrability of \(V\); indeed, we have \(V \in H^1(\Omega)\) and \(V \ge 0\) almost everywhere in \(\Omega\). Applying the method of sub and supersolutions with subsolution \(0\) and supersolution \(\NP\mu\), we conclude that the Dirichlet problem~\eqref{eq:dirichletProblem} with datum \(\mu\) has a nonnegative solution.

In dimension \(N = 2\), every measure is diffuse with respect to \(\capt_{\ch}\) by the counterpart of \Cref{prop:equivalence} in this dimension. We then rely on the Brezis--Merle inequality \cite{Brezis_Merle:1991}*{Theorem~1} to deduce that \(\Exp^{\cN\mu / \norm{\mu}_{\cM(\Omega)}} \in L^1(\Omega)\), where
\[
\NP\mu(x) = \frac{1}{2\pi} \int_{\R^2} \log \paren*{\frac{\diam \Omega}{\abs{x-y}}} \d\mu(y).
\]
The rest of the proof is unchanged.
\end{proof}

We may summarize the counterpart of \Cref{prop:strongApproximation} for the \(\ch\) capacity as follows:

\begin{lemma}
\label{lem:strongApproximation}
Let \(\mu \in \cM(\Omega)\) be a nonnegative measure. If \(\mu\) is diffuse with respect to the \(\ch\) capacity, then there exists a nondecreasing sequence \((\mu_n)_{n \in \N}\) of nonnegative measures in \(\cM(\Omega)\) with compact support in \(\Omega\) which satisfies
\begin{enumerate}[(i)]
\item for every \(n \in \N\), there exists \(C_n > 0\) such that, for every \(\varphi \in C_c^\infty(\R^N)\),
\[
\abs*{\int_{\Omega} \varphi \d\mu_n} \le C_n \norm{\Delta \varphi}_{H^1(\R^N)};
\]
\item \((\mu_n)_{n \in \N}\) converges strongly to \(\mu\) in \(\cM(\Omega)\).
\end{enumerate}
\end{lemma}

\resetconstant

\begin{proof}[Proof of \Cref{lem:strongApproximation}]
For every \(\varphi \in C_c^\infty(\R^N)\), we have the weak capacitary inequality
\begin{equation}
\label{eq:hardyWeakCapacitaryEstimate}
\capt_{\ch}(\set*{\abs{\varphi} \ge 1}) \le C \norm{\Delta \varphi}_{H^1(\R^N)}.
\end{equation}
Indeed, combining \eqref{eq:secondOrderLaplacian} and \eqref{eq:boxingInequality}, we obtain
\[
\haus_{\infty}^{N-2}(\set*{\abs{\varphi} \ge 1}) \le \C \norm{\Delta \varphi}_{H^1(\R^N)},
\]
and then \eqref{eq:hardyWeakCapacitaryEstimate} follows from the equivalence between \(\haus_{\infty}^{N-2}\) and \(\capt_{\ch}\) (\Cref{prop:equivalence}). 
The rest of the proof follows along the lines of the proof of \Cref{prop:strongApproximation}.
\end{proof}

Every nonnegative measure \(\mu \in \cM(\Omega)\) satisfying the functional inequality in Assertion~{\itshape (i)} also satisfies the estimate
\begin{equation}
\label{eq:strongEstimate}
\mu \le C \capt_{\ch}.
\end{equation}
Indeed, given a compact set \(K \subset \R^N\), let \((\varphi_n)_{n \in \N}\) be a sequence of functions admissible in the definition of the \(\ch\) capacity of \(K\) such that
\[
\lim_{n \rightarrow \infty} \norm{\Delta \varphi_n}_{H^1(\R^N)} = \capt_{\ch}(K).
\]
For each \(n \in \N\), we have
\[
\mu(K) \le \int_{\Omega} \varphi_n \d\mu \le C \norm{\Delta \varphi_n}_{H^1(\R^N)}.
\]
Letting \(n\) tend to infinity, we obtain \eqref{eq:strongEstimate}.

Assuming that the measure \(\mu\) satisfies the functional estimate
\[
\abs*{\int_{\Omega} \psi \d\mu} \le C \norm{\Delta \psi}_{H^1(\R^N)},
\]
for every smooth function \(\psi : \R^{N} \to \R\) that converges sufficiently fast to zero at infinity, in the spirit of the conclusion of \Cref{lem:strongApproximation}, one can prove that \(V \NP\mu \in L\loc^1(\Omega)\) based on the duality between the Hardy space \(H^1\) and the space \(\BMO\) of functions of bounded mean oscillation. 
Indeed, denoting by \(F\) the fundamental solution of the Laplacian in dimension \(N \ge 3\), we first write the Newtonian potential generated by \(\mu\) as \(\NP \mu = F * \mu\). For every \(\varphi \in C_c^\infty(\R^N)\) such that \(\displaystyle \int_{\R^N} \varphi = 0\), we then have the estimate
\[
\abs*{\int_{\R^N} \varphi \NP \mu} = \abs*{\int_{\Omega} F * \varphi \d\mu} \le C \norm{\Delta (F * \varphi)}_{H^1(\R^N)} = C \norm{\varphi}_{H^1(\R^N)}.
\]
By the density of such functions \(\varphi\) in the Hardy space \(H^1\), we deduce from the estimate above that the Newtonian potential \(\NP\mu\) has a unique extension as a continuous linear functional in the dual space of \(H^1\), which is precisely the space \(\BMO\) by Fefferman's characterization \citelist{\cite{Fefferman:1971} \cite{Stein:1993}}. Since \(V \NP\mu \ge 0\) almost everywhere in \(\Omega\), one has \(V \NP\mu \in L\loc^1(\Omega)\); see e.g.~\cite{Bonami_Iwaniec_Jones_Zinsmeister:2007}.

For a nonnegative measure \(\mu \in \cM(\Omega)\) satisfying the weaker assumption \(\mu \le C \capt_{\ch}\), or equivalently \(\mu \le C' \haus_{\infty}^{N-2}\), the fact that the Newtonian potential generated by \(\mu\) has bounded mean oscillation in dimension \(N \ge 3\) can be deduced from \cite{Adams:1975}*{Proposition~3.3}; see also \cite{Ponce:2016}*{Proposition~17.3}.

We finally turn to the

\begin{proof}[Proof of \Cref{thm:existence}]
Since the equation is linear, we may assume that \(\mu\) is nonnegative; see the proof of the converse of \Cref{thm:characterization} above. 
Let \((\mu_i)_{i \in \N}\) be a sequence of nonnegative measures in \(\cM(\Omega)\) with compact support in \(\Omega\) satisfying the conclusion of \Cref{lem:strongApproximation}. 
By \Cref{prop:existenceCompactSupport}, for each \(i \in \N\) the Dirichlet problem~\eqref{eq:dirichletProblem} with datum \(\mu_i\) has a nonnegative solution \(u_i\). It then suffices to proceed as in the proof of the converse of \Cref{thm:characterization} and conclude that the Dirichlet problem~\eqref{eq:dirichletProblem} with datum \(\mu\) has a solution \(u\), obtained as the limit in \(L^1(\Omega)\) of the sequence \((u_i)_{i \in \N}\).
\end{proof}

\section{A strong maximum principle in terms of the \(\ch\) capacity}

\Cref{thm:existence} can be applied to deduce a strong maximum principle for nonnegative potentials in \(H^1(\Omega)\):

\begin{proposition}
\label{prop:strongMaximumPrinciple}
Suppose that \(\Omega\) is connected and \(V \in H^1(\Omega)\) is a nonnegative function. If \(u \in L^1(\Omega)\) is a nonnegative function such that \(Vu \in L^1(\Omega)\) and
\[
-\Delta u + Vu \ge 0 \quad \text{in the sense of distributions in \(\Omega\)},
\]
and if the average integral of \(u\) satisfies
\[
\lim_{r \rightarrow 0} \fint_{B(x;r)} u = 0,
\]
for every point \(x\) in a compact subset of positive \(\ch\) capacity, then \(u = 0\) almost everywhere in \(\Omega\).
\end{proposition}

This is a counterpart in the Hardy space \(H^1\) setting of the strong maximum principle proved by Ancona~\cite{Ancona:1979} in \(L^1\) and by Orsina and Ponce~\cite{Orsina_Ponce:2016} in \(L^p\) for \(p > 1\). Compared to their results, we require the potential \(V\) to be nonnegative, since \(V^+\) need not belong to \(H^1(\Omega)\) if \(V \in H^1(\Omega)\).

We sketch the proof of \Cref{prop:strongMaximumPrinciple} in the case where \(u\) is a smooth function up to the boundary. The proof in full generality can be implemented along the lines of the proof of Theorem~1 in \cite{Orsina_Ponce:2016}.

\begin{proof}[Sketch of the proof of \Cref{prop:strongMaximumPrinciple} when \(\boldsymbol{u \in C^\infty(\overline{\Omega})}\)]
Let \(K \subset \Omega\) be a compact set such that \(K \subset \set{u = 0}\) and \(\capt_{\ch}(K) > 0\). 
Using the Riesz representation theorem and the Hahn--Banach theorem, one deduces along the lines of the proof of Proposition~A.17 in \cite{Ponce:2016} that there exists a positive measure \(\mu \in \cM(\Omega)\) supported in \(K\) such that
\[
0 \le \int_{K} \varphi \d\mu \le C \norm{\Delta \varphi}_{H^1(\R^N)},
\]
for every nonnegative function \(\varphi \in C_c^\infty(\R^N)\). In particular, \(\mu\) is diffuse with respect to \(\capt_{\ch}\). By \Cref{thm:existence}, the Dirichlet problem~\eqref{eq:dirichletProblem} with datum \(\mu\) has a nonnegative solution \(v\). One can then find a function \(f \in L^\infty(\Omega)\) explicitly defined in terms of \(v\) such that \(f > 0\) almost everywhere in \(\Omega\), and the solution of the Dirichlet problem
\[
\left\{
\begin{alignedat}{2}
-\Delta w + Vw &= \mu - f && \quad \text{in \(\Omega\)}, \\
w &= 0 && \quad \text{on \(\partial\Omega\)},
\end{alignedat}
\right.
\]
is nonnegative \cite{Ponce:2016}*{Lemma~22.12}. 
Since \(u \ge 0\) on \(\partial \Omega\), one proves that
\[
- \int_{\Omega} uf = \int_{\Omega} u(\mu - f) \ge \int_{\Omega} w (-\Delta u + Vu).
\]
Observe that the integral in the right-hand side is nonnegative.
By the nonnegativity of \(uf\), we deduce that \(uf = 0\) almost everywhere in \(\Omega\). Thus, \(u = 0\) in \(\Omega\) and this concludes the proof of the proposition.
\end{proof}

\section{Proof of \texorpdfstring{\Cref{thm:existenceNondiffuse}}{Theorem \ref{thm:existenceNondiffuse}} and generalization to Orlicz spaces}

In this section, we prove a stronger statement which implies \Cref{thm:existenceNondiffuse}. For this purpose, we recall that the Orlicz space \(L\loc^{\Phi}(\Omega)\) is the vector space spanned by the set
\[
\set[\bigg]{u : \Omega \rightarrow \R : \text{\(u\) is measurable and \(\int_{\omega} \Phi(\abs{u}) < \infty\), for every \(\omega \Subset \Omega\)}},
\]
where \(\Phi : \left[ 0,\infty \right) \rightarrow \R\) is a continuous convex function such that
\[
\lim_{s \rightarrow 0} \frac{\Phi(s)}{s} = 0 \quad \text{and} \quad \lim_{s \rightarrow \infty} \frac{\Phi(s)}{s} = \infty.
\]
In particular, such a function \(\Phi\) is nondecreasing and satisfies \(\Phi(0) = 0\).

\begin{proposition}
\label{prop:existenceNondiffuseOrlicz}
Suppose that \(N \ge 3\). 
For every \(\Phi : \left[ 0,\infty \right) \rightarrow \R\) as above, there exists a positive measure \(\mu \in \cM(\Omega)\) with \(\haus^{N-2}(\supp \mu) = 0\) such that the Dirichlet problem~\eqref{eq:dirichletProblem} with datum \(\mu\) has a solution for every nonnegative function \(V \in L\loc^{\Phi}(\Omega) \cap L^1(\Omega)\).
\end{proposition}

The main ingredient in the proof of \Cref{prop:existenceNondiffuseOrlicz} is a construction from \cite{Ponce:2005}. We summarize the main facts that are used hereafter. In dimension \(N \ge 3\), one shows that for every continuous nondecreasing function \(g : \left[ 0, \infty \right) \rightarrow \R\) with \(g(0) = 0\) there exists a positive measure \(\mu \in \cM(\Omega)\) with compact support in \(\Omega\) such that \(\haus^{N-2}(\supp \mu) = 0\) and 
\begin{equation}
\label{eq:gNewtonianPotential}
g(\NP\mu) \in L^1(\Omega).
\end{equation}
Indeed, the choice of the measure \(\mu\) and its support is made in the proof of Theorem~3 in \cite{Ponce:2005}; the proofs of Propositions~1~and~4 in that paper imply that \(g(\NP\mu) \in L^1(\Omega)\).

\begin{proof}[Proof of \Cref{prop:existenceNondiffuseOrlicz}]
Take \(g = \Phi^*\) to be the Legendre transform of \(\Phi\), defined for \(t \in \left[ 0,\infty \right)\) by
\[
\Phi^*(t) = \sup_{s \ge 0} \set*{st - \Phi(s)}.
\]
Let \(\mu \in \cM(\Omega)\) be the positive measure that satisfies \eqref{eq:gNewtonianPotential} above, and let \(V \in L\loc^{\Phi}(\Omega) \cap L^1(\Omega)\) be a nonnegative function. Since the Newtonian potential \(\NP\mu\) is nonnegative, it follows from Young's inequality that
\[
V \NP\mu \le \Phi(V) + \Phi^*(\NP\mu) \quad \text{almost everywhere in \(\Omega\)}.
\]
Since \(\Phi(V) \in L\loc^1(\Omega)\) and \(\NP\mu\) is harmonic in a neighborhood of \(\partial \Omega\), we have \(V \NP\mu \in L^1(\Omega)\).  
Hence, \(\NP\mu\) is a supersolution of the Dirichlet problem~\eqref{eq:dirichletProblem} with datum \(\mu\).{}
We then conclude as in the proof of \Cref{prop:existenceCompactSupport} using the method of sub and supersolutions.
\end{proof}

We now explain how one can deduce \Cref{thm:existenceNondiffuse} from \Cref{prop:existenceNondiffuseOrlicz}. For this end, let \(\Phi : \left[ 0,\infty \right) \rightarrow \R\) be the function defined for \(s \in \left[ 0,\infty \right)\) by
\[
\Phi(s) = s \log s - s + 1.
\]
This function defines an Orlicz space, and if \(V \in H^1(\Omega)\) is a nonnegative function, we have
\[
V \in L\loc^{\Phi}(\Omega).
\]
Observe that the Legendre transform of \(\Phi\) is given by
\[
\Phi^*(t) = \Exp^t - 1,
\]
for every \(t \ge 0\). In this case, we have an example of a measure \(\mu \in \cM(\Omega)\) such that
\[
\Exp^{\NP\mu} \in L^1(\Omega),
\]
but which cannot be approximated strongly by measures such that \(\nu \le C \capt_{\ch}\).{}
Indeed, any such a measure \(\nu\) satisfies \(\nu(\supp{\mu}) = 0\), hence is singular with respect to \(\mu\).

\section{Nonexistence of solutions}
\label{sec:nonexistence}

This last section is dedicated to nonnexistence results for some suitably chosen measure data depending on the potential \(V\). 
We begin with the example given in the introduction, namely

\begin{proposition}
\label{prop:nonexistenceExample}
Suppose that \(N \ge 3\). For every \(\alpha \ge 2\), the equation
\[
-\Delta u + \frac{u}{\abs{x}^{\alpha}} = \delta_0 \quad \text{in \(B_1\)}
\]
has no solution in the sense of distributions.
\end{proposition}

\begin{proof}
Assume by contradiction that \(u\) is a solution of the equation above in the sense of distributions. In particular, \(\frac{u}{\abs{x}^{\alpha}} \in L^1(B_1)\). On the one hand, since \(-\Delta u\) has a Dirac mass at the origin, for every \(0 < r < 1\) the average integral of \(u\) over the sphere \(\partial B_r\) satisfies \cite{Ponce:2016}*{Lemma~21.4}
\[
\lim_{r \rightarrow 0} r^{N-2} \fint_{\partial B_r} u \d\sigma = \frac{1}{(N-2)\sigma_N}.
\]
On the other hand, by the integration formula in polar coordinates, we have
\[
\int_{B_1} \frac{u}{\abs{x}^{\alpha}} = \int_0^1 \paren[\bigg]{\int_{\partial B_r} \frac{u}{r^{\alpha}} \d\sigma} \d{r} = \sigma_N \int_0^1 \frac{1}{r^{\alpha - N + 1}} \paren[\bigg]{\fint_{\partial B_r} u \d\sigma} \d{r}.
\]
Take \(0 < \varepsilon < 1\) such that, for every \(0 < r < \varepsilon\),
\[
\fint_{\partial B_r} u \d\sigma \ge \frac{1}{2(N-2)\sigma_N} \frac{1}{r^{N-2}}.
\]
Then,
\[
\int_{B_1} \frac{u}{\abs{x}^{\alpha}} \ge C \int_0^{\varepsilon} \frac{1}{r^{\alpha-1}} \d{r} = \infty,
\]
which is a contradiction.
\end{proof}

The previous result can be pursued to measures supported on larger sets, for example, on manifolds of dimension \(N-2\). 
Those sets have zero \(\cl{1}\) capacity, but positive \(\ch\) capacity by \Cref{prop:equivalence}.

For a manifold \(M \subset \Omega\) and \(r > 0\), we denote its tubular neighborhood of radius \(r\) by
\[
\Nu_r = \set*{x \in \R^N : \dist(x,M) < r}.
\]
We also denote its annular tubular neighborhood of inner radius \(\theta r\) and outer radius \(r\) by
\[
\Lambda_{r,\theta} = \Nu_r \setminus \overline{\Nu_{\theta r}},
\]
for some fixed \(0 < \theta < 1\). We prove

\begin{proposition}
\label{prop:nonexistence}
Suppose that \(N \ge 3\) and \(M \subset \Omega\) is a compact and smooth manifold without boundary of dimension \(N - 2\). Let \(0 < \delta < \dist(M,\partial \Omega)\) and \(V \in L^1(\Omega)\) be a nonnegative function such that, for every \(0 < r < \delta\), 
\[
V \le C \fint_{\Lambda_{r,\theta}} V \quad \text{almost everywhere in \(\Lambda_{r,\theta}\)},
\]
where \(\displaystyle\fint_{\Lambda_{r,\theta}} V\) denotes the average integral of \(V\) over \(\Lambda_{r,\theta}\). 
If we have
\[
\int_0^{\delta} \paren[\bigg]{\abs{\log r} \int_{\Lambda_{r,\theta}} V} \frac{\d{r}}{r} = \infty,
\]
then the equation
\[
- \Delta u + Vu = \haus^{N-2} \lfloor_M \quad \text{in \(\Omega\)}
\]
has no solution in the sense of distributions.
\end{proposition}

For the proof of \Cref{prop:nonexistence}, we rely on a computation from \cite{Davila_Ponce:2003}*{Theorem~6} asserting that if \(v \in L\loc^1(\Omega)\) satisfies for some \(g \in L\loc^1(\Omega)\) the equation
\[
\Delta v = g \quad \text{in the sense of distributions in \(\Omega\)},
\]
then, for every smooth compact manifold \(M \subset \Omega\) without boundary of dimension \(N-2\), one has
\begin{equation}
\label{eq:davilaPonce}
\lim_{r \rightarrow 0} \frac{1}{r^2 \abs{\log r}} \int_{\Nu_r} \abs{v} = 0.
\end{equation}

\resetconstant

\begin{proof}[Proof of \Cref{prop:nonexistence}]
Assume by contradiction that \(u\) is a solution of the equation above and let \(v\) be the solution of the Dirichlet problem
\[
\left\{
\begin{alignedat}{2}
-\Delta v &= Vu && \quad \text{in \(\Omega\)}, \\
v &= 0 && \quad \text{on \(\partial\Omega\)}.
\end{alignedat}
\right.
\]
Throughout the proof, we use the notation \(\mu = \haus^{N-2} \lfloor_M\).
Since
\[
-\Delta u = -\Delta (-v + \NP\mu) \quad \text{in the sense of distributions in \(\Omega\)},
\]
by Weyl's lemma there exists a harmonic function \(h : \Omega \rightarrow \R\) such that
\[
u = - v + \NP\mu + h \quad \text{almost everywhere in \(\Omega\)}.
\]
Taking \(\delta\) to be smaller if necessary, there exists \(\Cl{cst:davila_ponce} > 0\) such that
\[
\NP\mu \ge C_1 \abs{\log r} \quad \text{almost everywhere in \(\Lambda_{r,\theta}\)},
\]
for every \(0 < r < \delta\). Thus,
\[
\int_{\Lambda_{r,\theta}} V \NP\mu \ge \Cs \abs{\log r} \int_{\Lambda_{r,\theta}} V.
\]
Let \(\varepsilon > 0\). Taking \(\delta\) to be smaller if necessary, by \eqref{eq:davilaPonce} applied to \(v - h\), we have
\[
\int_{\Lambda_{r,\theta}} \abs{v - h} \le \varepsilon r^2 \abs{\log r},
\]
for every \(0 < r < \delta\). Using the upper bound of \(V\) in terms of its average integral, we get
\[
\int_{\Lambda_{r,\theta}} V \abs{v - h} \le C \int_{\Lambda_{r,\theta}} \abs{v - h} \fint_{\Lambda_{r,\theta}} V \le C \varepsilon r^2 \abs{\log r} \fint_{\Lambda_{r,\theta}} V \le \C \varepsilon \abs{\log r} \int_{\Lambda_{r,\theta}} V,
\]
whence
\[
\int_{\Lambda_{r,\theta}} Vu \ge (C_1 - \Cs \varepsilon) \abs{\log r} \int_{\Lambda_{r,\theta}} V.
\]
Dividing both sides by \(r\) and integrating from \(0\) to \(\delta\) with respect to \(r\), we obtain
\[
\int_0^\delta \paren[\bigg]{\int_{\Lambda_{r,\theta}} Vu} \frac{\d{r}}{r} \ge (\Cr{cst:davila_ponce} - \Cs \varepsilon) \int_0^\delta \paren[\bigg]{\abs{\log r} \int_{\Lambda_{r,\theta}} V} \frac{\d{r}}{r}.
\]
Taking \(\varepsilon > 0\) such that \(\Cr{cst:davila_ponce} > \Cs \varepsilon\), the inequality above yields
\[
\int_0^\delta \paren[\bigg]{\int_{\Lambda_{r,\theta}} Vu} \frac{\d{r}}{r} = \infty.
\]
On the other hand, we deduce from Fubini's theorem that
\[
\begin{split}
\int_0^\delta \paren[\bigg]{\int_{\Lambda_{r,\theta}} Vu } \frac{\d{r}}{r} 
& = \int_{\Nu_{\delta}} \paren[\bigg]{\int_{\dist(x,M)}^{\min \set{\dist(x,M)/\theta,\delta}} \frac{\d{r}}{r}} V(x)u(x) \d{x} \\ 
& \le \int_{\Nu_{\delta}} \paren[\bigg]{\int_{\dist(x,M)}^{\dist(x,M)/\theta} \frac{\d{r}}{r}} V(x)u(x) \d{x} 
= \paren[\bigg]{\log \frac{1}{\theta}} \int_{\Nu_{\delta}} Vu.
\end{split}
\]
Thus, \(Vu \not\in L^1(\Nu_{\delta})\), which contradicts the assumption that \(u\) satisfies the equation with potential \(V\) and concludes the proof of the proposition.
\end{proof}

To illustrate the connection between \Cref{prop:nonexistence} and the Dirichlet problem involving \(H^{1}\) potentials, we sketch the construction of a signed function \(f \in H^1(\Omega)\) such that \(V = f^+\) satisfies the assumptions of \Cref{prop:nonexistence}. For this purpose, we rely on Whitney decomposition's of \(\Omega \setminus M\) (cf.~\Cref{fig:whitneyDecomposition} below) and the atomic characterization of the Hardy space \(H^1\); see \cite{Stein:1993}*{Chapter~III}.

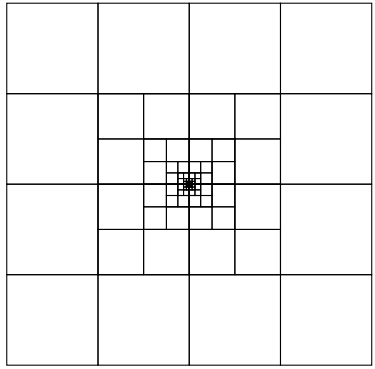
\begin{figure}[h]
\begin{tikzpicture}[scale=1.2]
\foreach \i in {0,...,6} { \foreach \j in {0,...,3} { \foreach \k in {0,...,3} {
\draw ({2*(1-0.5^\i))+\j/(2^\i)},{2*(1-0.5^\i))+\k/(2^\i)}) rectangle ({2*(1-0.5^(\i+1)))+\j/(2^\i)},{2*(1-0.5^(\i+1)))+\k/(2^\i)}); }}}
\end{tikzpicture}
\caption{Whitney's decomposition of \(\bracks*{-\frac{1}{2},\frac{1}{2}}^2 \setminus \set{0}\).}
\label{fig:whitneyDecomposition}
\end{figure}

\begin{corollary}
Suppose that \(N \ge 3\). Then there exists \(f \in H^1(\Omega)\) such that the Dirichlet problem~\eqref{eq:dirichletProblem} with potential \(V = f^+\) does not have a solution for every measure \(\mu \in \cM(\Omega)\) which is diffuse with respect to the \(\ch\) capacity.
\end{corollary}

\resetconstant

\begin{proof}
For convenience, we may assume that \(\Omega\) is a unit open cube. Let \(M \subset \Omega\) be a smooth compact manifold of dimension \(N-2\). By the Whitney decomposition theorem \cite{Stein:1970}*{Chapter~I, Theorem~3}, there exists a family \((Q_i)_{i \in \N}\) of closed cubes with pairwise disjoint interiors such that
\[
\bigcup_{i \in \N} Q_i = \Omega \setminus M,
\]
and
\begin{equation}
\label{eq:whitneyDistanceDiameter}
\diam Q_i \le \dist(Q_i, M) \le 4 \diam Q_i,
\end{equation}
for every \(i \in \N\). For each \(i \in \N\), denote by \(p_i\) the center of the cube \(Q_i\), and by \(l_i\) its side length. Fix a function \(a : \R^N \rightarrow \R\) supported in the cube \(\bracks{-\frac{1}{2}, \frac{1}{2}}^N\) such that \(a = 1\) in the upper half of the cube, and \(a = -1\) in the lower half. Since
\[
\int_{\R^N} a = 0 \quad \text{and} \quad \abs{a} \le 1 \quad \text{in \(\R^N\)},
\]
\(a\) is an \(H^1\) atom. The function \(a_i : \R^N \rightarrow \R\) defined for \(x \in \R^N\) by
\[
a_i(x) =  \frac{1}{\abs{Q_i}} \, a \paren*{\frac{x - p_i}{l_i}}
\]
is also an \(H^1\) atom, supported by \(Q_i\). We now gather the cubes \(Q_i\) in disjoint classes \(\cF_j\), with \(j \in \N \setminus \set{0}\): we say that \(Q_i \in \cF_j\) if
\[
\dist(p_i, M) \sim \frac{1}{2^j}.
\]
Since \(M\) is a manifold, the number of cubes in \(\cF_j\) is bounded from above, independently of \(j\). Given a summable sequence \((\alpha_j)_{j \in \N \setminus \set{0}}\) of positive numbers, it follows from the atomic characterization of \(H^1(\R^N)\) that the function
\[
f = \sum_{j=1}^\infty \alpha_j \paren[\Bigg]{\sum_{a_i \in \cF_j} a_i}
\]
belongs to \(H^1(\R^N)\) and
\[
\norm{f}_{H^1(\R^N)} \le C \sum_{j=1}^\infty \alpha_j.
\]
Assume further that
\begin{equation}
\label{eq:assumptionSequence}
\C \le \frac{\alpha_j}{\alpha_{j+1}} \le \C,
\end{equation}
with positive constants \(C_1\) and \(C_2\) independent of \(j\). Take the factor \(0 < \theta < 1\) such that, given \(0 < r < \dist(M, \partial\Omega)\), if \(j \in \N \setminus \set{0}\) is the smallest integer such that
\[
\bigcup_{a_i \in \cF_j} a_i \subset \Nu_r,
\]
then
\begin{equation}
\label{eq:familySubset}
\bigcup_{a_i \in \cF_j} a_i \subset \Lambda_{r,\theta}.
\end{equation}
Such a \(\theta\) exists by virtue of \eqref{eq:whitneyDistanceDiameter} and can be explicitly estimated in terms of the dimension~\(N\). Using \eqref{eq:assumptionSequence} and \eqref{eq:familySubset}, one verifies that
\[
\int_{\Lambda_{r,\theta}} f^+ \sim \int_{\Lambda_{r,\theta}} \abs{f} \sim \alpha_j,
\]
and then
\[
\abs{\log r} \int_{\Lambda_{r,\theta}} f^+ \sim j \alpha_j,
\]
since \(r \sim 1/2^{j}\).
Assuming for simplicity that \(\Nu_1 \subset \Omega\), then \(\displaystyle \bigcup_{j \in \N \setminus \set{0}} \paren[\bigg]{\bigcup_{a_i \in \cF_j} a_i} \subset \Omega\) and we have
\[
\int_0^1 \paren[\bigg]{\abs{\log r} \int_{\Lambda_{r,\theta}} f^+} \frac{\d{r}}{r} \sim \sum_{j=1}^\infty j \alpha_j.
\]
Taking \(\alpha_j = {1}/{j^2}\), the function \(V = f^+\) satisfies the assumptions of \Cref{prop:nonexistence}, and the conclusion follows with \(\mu = \haus^{N-2} \lfloor_M\).
\end{proof}

\section*{Acknowledgements}

The authors would like to thank D.~Spector for bringing to their attention D.~Adams' equivalence between the \(\ch\) capacity and the \(\haus^{N-2}_{\infty}\) Hausdorff content, and Q.-H. Nguyen for discussions on the strong maximum principle for the Schrödinger operator.
The first author (ACP) was supported by the Fonds de la Recherche scientifique -- FNRS under research grant J.0026.15.

\end{document}